\documentclass[10pt]{amsart}
\usepackage[english]{babel}
\usepackage{amssymb}
\usepackage{amsmath}
\usepackage{srcltx}
\usepackage{lscape}

\newtheorem{dummy}{Dummy}

\newtheorem{lemma}[dummy]{Lemma}
\newtheorem{theorem}[dummy]{Theorem}
\newtheorem{proposition}[dummy]{Proposition}
\newtheorem{corollary}[dummy]{Corollary}

\theoremstyle{definition}
\newtheorem{definition}{Definition}

\newtheorem{example}[dummy]{Example}

\newtheorem{remark}[dummy]{Remark}

\newcommand{\ignore}[1]{}

\date{14.7.2016}
\author{S. Pumpl\"un}
\email{susanne.pumpluen@nottingham.ac.uk}
\address{School of Mathematical Sciences\\
University of Nottingham\\
University Park\\
Nottingham NG7 2RD\\
United Kingdom
}

\keywords{Differential polynomial ring, skew polynomial, differential polynomial, differential operator,
 differential algebra, nonassociative division algebra.}

\subjclass[2010]{Primary: 17A35; Secondary: 17A60, 17A36}

\begin{document}

\title[Nonassociative differential extensions of characteristic $p$]
{Nonassociative differential extensions of characteristic $p$}

\begin{abstract} Let $F$ be a field of characteristic $p$. We define and investigate
nonassociative differential extensions of $F$ and of a central simple division algebra over $F$
and give a criterium for these algebras to be division. As special cases, we obtain classical results
for associative algebras by Amitsur and Jacobson.
We construct families of nonassociative division algebras which can be viewed as
 generalizations of associative cyclic extensions of a purely inseparable field extension of exponent one or a central division algebra.
 Division algebras which are nonassociative cyclic extensions of a purely inseparable field extension of exponent one
 are particularly easy to obtain.
\end{abstract}

\maketitle

%
\section*{Introduction}
%

Differential polynomial rings $D[t;\delta]$, where $D$ is a division algebra
over a field $F$ and $\delta$ a derivation on $D$, have been used successfully to construct associative
 central simple algebras.
These appear either as a quotient algebra $D[t;\delta]/(f)$ when factoring out a two-sided ideal
generated by a differential polynomial $f\in D[t;\delta]$, cf. \cite{Am2},  \cite{Hoe},
\cite[Sections 1.5, 1.8, 1.9]{J96},  or as the eigenring
of a differential polynomial $f$, e.g. see  \cite{Am}, \cite{O}. We can put these constructions into a
nonassociative context as follows:

Given $f\in D[t;\delta]$ of degree $m$, the set of all differential polynomials of degree less than $m$ can be
 canonically equipped with a nonassociative ring structure,
 using right division by $f$ to define the  multiplication $g\circ h=gh \,\,{\rm mod}_r f $.
The resulting nonassociative unital ring $S_f$
is an algebra over the field  $F_0=C(D)\cap {\rm Const}(\delta)$  (Petit \cite{P66}).
If $f$ is not two-sided (i.e., does not generate a two-sided ideal in $D[t;\delta]$) and $\delta$ is not trivial, then the $S_f$ are algebras whose nuclei are larger than their
center $F_0$. In particular, their right nucleus is the eigenring of $f$ employed in \cite{Am} and \cite{O}, whereas if
$f$ generates
 a two-sided ideal, then $S_f$ is the (associative) quotient algebra employed in  \cite{Am2} and \cite{J96},
 each time for well considered choices of $f$ and $D[t;\delta]$.

Let $F$ be a field of characteristic $p>0$. We study the  algebras $S_f$ containing
 a purely inseparable field extension  $K/F$ of exponent one or a central division algebra $D$
 over  $F$ as left nucleus.
  As a special case we reprove the classical results on differential extensions by  Jacobson \cite{J96}
  and Amitsur's results on noncommutative cyclic extensions of degree $p$ \cite{Am2}.

The paper is organized as follows: we introduce the basic terminology in Section \ref{sec:prel}. In Section \ref{sec:2}
we focus on the case that $\delta$
is a quasi-algebraic derivation with minimal polynomial $g$ and therefore $S_f$ an algebra of finite
dimension over $F={\rm Const}(\delta)$.
In particular, for $f(t)=g(t)-d\in D[t;\delta]$ where $g$ is the minimum polynomial of $\delta|_{C(D)}$, the set of all logarithmic derivatives
$\{\delta(c)/c\,|\, c\in C(D) \}$ turns out to be a subgroup of the automorphism group of $S_f$.
We follow up on this observation and  define
 nonassociative differential extensions of a field in Section \ref{sec:3}  and
nonassociative differential extensions of a central simple division algebra in Section \ref{sec:4},
generalizing classical constructions by Amitsur and Jacobson, by choosing $f(t)=g(t)-d\in D[t;\delta]$
to be a $p$-polynomial of a certain type.

In particular, when $K$ is a purely inseparable extension of $F$ of exponent one with derivation $\delta$ such that
$\delta$ has minimum polynomial $g(t)=t^p-t\in F[t]$ and
 $f(t)=t^p-t-d\in K[t;\delta]$,
 $S_f=(K,\delta,d)$ is a unital nonassociative division algebra over $F={\rm Const}(\delta)$ of dimension $p^{2}$ for
 all $d\in K\setminus F$.
 Its automorphism group contains a cyclic subgroup of order $p$
which leaves $K$ invariant (Theorem \ref{thm:strongest}). This canonically generalizes Amitsur's associative cyclic
extensions of degree $p$. Thus $(K(x),\delta,h(x))$ is a division algebra over $F(x)$
of dimension $p^{2}$  for all
$h(x)\in K(x)\setminus F(x)$, and so a nonassociative cyclic extension of $K(x)$ (Example \ref{ex:1}).
This generalizes \cite[Proposition 1.9.10]{J96}.
Analogously, Theorem \ref{thm:last} in Section \ref{sec:4} generalizes
the result on associative cyclic extensions of $D$, cf. \cite[Theorem 1.3.27]{J96}.

In Section \ref{sec:tensor} we construct tensor products of a central simple division algebra and a
nonassociative cyclic
extension, generalizing another classical result by Jacobson \cite[Theorem 1.9.13]{J96}
in Theorem \ref{thm:Jacobsontensor}. As an application, we show that
$(K(x),\delta,h(x))\otimes_{F(x)}D_{F(x)}$ with $h(x)\in K(x)\setminus F(x)$ is a division algebra if and only if
$h(x)\not=(t-z)^p-t^p-z$ for all $z\in D_{K(x)}$
in Example \ref{ex:last}, provided that $\delta$ has minimum polynomial $g(t)=t^p-t$ and that
$D\otimes_F K$ is a division algebra. This algebra is a nonassociative cyclic
extension of $D_{K(x)}$ if it is division.
This generalizes \cite[Theorem 1.9.11]{J96}, where $h(x)=x$ in which case the algebra is division.
\\\\
The theory presented in this paper can be extended to nonassociative cyclic extensions of degree any prime power if desired, along
the lines presented here. It complements the theory of nonassociative cyclic algebras $(K/F,\sigma,d)$ which are constructed
out of twisted polynomial rings $K[t;\sigma]$ and $f(t)=t^m-d\in K[t;\sigma]$, where $K/F$ is a cyclic Galois extension of degree $m$,
${\rm Gal}(K/F)=<\sigma>$ and $F$ has characteristic zero or $p$, but now with $p$ coprime to $m$, cf. \cite{S12}, and the
theory of nonassociative generalized cyclic algebras $(D,\sigma,d)$ which are constructed
out of twisted polynomial rings $D[t;\sigma]$ and $f(t)=t^m-d\in D[t;\sigma]$, where $D$ is a cyclic division algebra
 of degree $m$, $f(t)=t^m-d\in D[t;\sigma]$, and $\sigma$ chosen suitably, cf. \cite{P16.1}, \cite{PS15.4}.

%
%

\section{Preliminaries} \label{sec:prel}

\subsection{Nonassociative algebras} \label{subsec:nonassalgs}


Let $F$ be a field and let $A$ be an $F$-vector space. $A$ is an
\emph{algebra} over $F$ if there exists an $F$-bilinear map $A\times
A\to A$, $(x,y) \mapsto x \cdot y$, denoted simply by juxtaposition
$xy$, the  \emph{multiplication} of $A$. An algebra $A$ is called
\emph{unital} if there is an element in $A$, denoted by 1, such that
$1x=x1=x$ for all $x\in A$. We will only consider unital algebras
from now on without explicitly saying so.

An algebra $A\not=0$ is called a \emph{division algebra} if for any
$a\in A$, $a\not=0$, the left multiplication  with $a$, $L_a(x)=ax$,
and the right multiplication with $a$, $R_a(x)=xa$, are bijective.
If $A$ has finite dimension over $F$, $A$ is a division algebra if
and only if $A$ has no zero divisors \cite[pp. 15, 16]{Sch}.


Associativity in $A$ is measured by the {\it associator} $[x, y, z] =
(xy) z - x (yz)$. The {\it left nucleus} of $A$ is defined as ${\rm
Nuc}_l(A) = \{ x \in A \, \vert \, [x, A, A]  = 0 \}$, the {\it
middle nucleus} of $A$ is ${\rm Nuc}_m(A) = \{ x \in A \, \vert \,
[A, x, A]  = 0 \}$ and  the {\it right nucleus} of $A$ is
${\rm Nuc}_r(A) = \{ x \in A \, \vert \, [A,A, x]  = 0 \}$. ${\rm
Nuc}_l(A)$, ${\rm Nuc}_m(A)$, and ${\rm Nuc}_r(A)$ are associative
subalgebras of $A$. Their intersection
 ${\rm Nuc}(A) = \{ x \in A \, \vert \, [x, A, A] = [A, x, A] = [A,A, x] = 0 \}$ is the {\it nucleus} of $A$.
${\rm Nuc}(A)$ is an associative subalgebra of $A$ containing $F1$
and $x(yz) = (xy) z$ whenever one of the elements $x, y, z$ is in
${\rm Nuc}(A)$. The
 {\it center} of $A$ is ${\rm C}(A)=\{x\in \text{Nuc}(A)\,|\, xy=yx \text{ for all }y\in A\}$.


\subsection{Differential polynomial rings}

Let $D$ be an associative division ring and $\delta:K\rightarrow K$ a \emph{derivation}, i.e. an
additive map such that $$\delta(ab)=a\delta(b)+\delta(a)b$$ for all $a,b\in K$.
 The \emph{differential polynomial ring} $D[t;\delta]$
is the set of polynomials $$a_0+a_1t+\dots +a_nt^n$$ with $a_i\in D$,
where addition is defined term-wise and multiplication by
$$ta=at+\delta(a) \quad (a\in K).$$
For $f=a_0+a_1t+\dots +a_nt^n$ with $a_n\not=0$ define ${\rm
deg}(f)=n$ and ${\rm deg}(0)=-\infty$. Then ${\rm deg}(fg)={\rm deg}(f)+{\rm deg}(g).$
 An element $f\in R$ is \emph{irreducible} in $R$ if it is no unit and  it has no proper factors,
 i.e if there do not exist $g,h\in R$ with ${\rm deg}(g),{\rm deg} (h)<{\rm deg}(f)$ such
 that $f=gh$.

 $R=D[t;\delta]$ is a left and right principal ideal domain and there is a right division algorithm in $R$: for all
$g,f\in R$, $g\not=0$, there exist unique $r,q\in R$ with ${\rm
deg}(r)<{\rm deg}(f)$, such that $$g=qf+r.$$
There is also a left division algorithm in $R$ \cite[p.~3 and Prop. 1.1.14]{J96}. (Our
terminology is the one used by Petit \cite{P66}; it is opposite to Jacobson's.)

  We know that
$$G_{\tau,a}(\sum_{i=0}^na_it^i)=\sum_{i=0}^n\tau(a_i)(t+a)^i$$ is an automorphism of $R=D[t;\delta]$ if and only if
$\tau$ is an automorphism of $D$ and
$$\delta(\tau(z))-\tau(\delta(z))=a\tau(z)-\tau(z)a$$
 for all $z\in D$ \cite{K}.

\subsection{Nonassociative algebras obtained from differential polynomial rings} \label{subsec:structure}

Let $f\in R=D[t;\delta]$ of degree $m$. Let ${\rm mod}_r f$ denote the remainder of right division by $f$.
Define  $F={\rm Cent}(\delta)=\{a\in D\,|\, \delta(a)=0\}$.

 \begin{definition} (cf. \cite[(7)]{P66})
  The vector space
$$R_m=\{g\in D[t;\delta]\,|\, {\rm deg}(g)<m\}$$
 together with the multiplication
 $$g\circ h=gh \,\,{\rm mod}_r f $$
 is a unital nonassociative algebra $S_f=(R_m,\circ)$ over
 $$F_0=\{a\in D\,|\, ah=ha \text{ for all } h\in S_f\}.$$
\end{definition}

 $F_0$  is a commutative subring of $D$ \cite[(7)]{P66} and it is easy to check that $F_0={\rm Cent}(\delta)\cap C(D)$.
 The algebra $S_f$ is also denoted by $R/Rf$ \cite{P66, P68}
 if we want to make clear which ring $R$ is involved in the construction.
 In the following, we call the algebras $S_f$  \emph{Petit algebras} and denote their multiplication simply by juxtaposition.

  Using left division by $f$ and
  the remainder ${\rm mod}_l f$ of left division by $f$, we can define
  a second unital nonassociative algebra structure on $R_m$ over $F$, called $\,_fS$ or $R/fR$.

It suffices to consider the Petit algebras $S_f$, however, since every algebra
$\,_fS$ is the opposite algebra of some Petit algebra (cf. \cite[(1)]{P66}).

 We call $f\in R$ a \emph{(right) semi-invariant} polynomial if for every $a\in D$ there is $b\in D$ such that
 $f(t)a=bf(t)$. If also $f(t)t=(ct+d)f(t)$ for some $c,d\in D$ then $f$ is called \emph{(right) invariant}.
 The invariant polynomials are also called \emph{two-sided}, as the ideals they generate are left and right ideals.

\begin{theorem} (cf. \cite[(2), p.~13-03, (5), (6), (7), (9)]{P66})  \label{Properties of S_f petit}
Let $f(t) 
\in R = D[t;\delta]$.
\\
(i) If $S_f$ is not associative then ${\rm Nuc}_l(S_f)={\rm Nuc}_m(S_f)=D$ and
$${\rm Nuc}_r(S_f)=\{g\in R\,|\, fg\in Rf\}.$$
 (ii) The powers of $t$ are associative if and only if $t^mt=tt^m$
 if and only if $t\in {\rm Nuc}_r(S_f)$ if and only if $ft\in Rf.$ 
\\ (iii) If $f$ is irreducible then ${\rm Nuc}_r(S_f)$ is an associative division algebra.
\\ (iv) Let $f\in R$ be irreducible and $S_f$ a finite-dimensional $F$-vector space
or free of finite rank as a right ${\rm Nuc}_r(S_f)$-module. Then $S_f$
is a division algebra.
\\ Conversely, if $S_f$ is a division algebra then $f$ is irreducible.
\\
(v) $S_f$ is associative if and only if $f$ is a two-sided element.
In that case, $S_f$ is the usual quotient algebra  $D[t;\delta]/(f)$.
\end{theorem}

\begin{proposition}\label{prop:subfield}
Let $R=D[t;\delta]$ and  $F_0=F\cap C(D)$. For all $f \in F_0[t;\delta]=F_0[t]$,
$$F_0[t]/(f)\cong F_0\oplus F_0t\oplus\dots\oplus F_0t^{m-1}$$ is a commutative subring of $S_f$ which is an
algebraic field extension of $F_0$ if $f(t) \in F_0[t]$ is irreducible, and
$$F_0[t]/(f)=F_0\oplus F_0t\oplus\dots\oplus F_0t^{m-1}\subset {\rm Nuc}_r(S_f).$$
\end{proposition}

\ignore{
By Theorem \ref{Properties of S_f petit}, the powers of $t$ are associative if and only if $t^mt=tt^m$
 if and only if $t\in {\rm Nuc}_r(S_f)$ if and only if $ft\in Rf.$
}

\begin{proof}
Since $f \in F_0[t;\delta]=F_0[t]$,
$S_f$ contains the commutative subring $F_0[t]/(f).$  This subring is isomorphic to the ring consisting of the elements $\sum_{i=0}^{m-1}a_it^i$
with $a_i\in F_0$. In particular, we know that the powers of $t$ are associative.
By Theorem \ref{Properties of S_f petit} (ii), this implies that $t\in {\rm Nuc}_r(S_f)$.
 Clearly $F_0\subset {\rm Nuc}_r(S_f)$, so if $t\in {\rm Nuc}_r(S_f)$  then
 $F_0\oplus F_0t\oplus\dots\oplus F_0t^{m-1}\subset  {\rm Nuc}_r(S_f)$, hence we obtain the assertion. If $f$ is irreducible in $F_0[t]$, this is an algebraic
field extension of $F_0$.
\end{proof}

\begin{proposition}  \label{prop:new(1)}
Let $f\in R$ be  of degree $m\geq2$.
Then $f$ is a semi-invariant polynomial if and only if
$$D\subset {\rm Nuc}_r(S_f).$$
\end{proposition}

\begin{proof}
 If $f\in R$ is a  semi-invariant polynomial then for every $a\in D$ there is $b\in D$ such that
 $f(t)a=bf(t)\in Rf$ and hence
 $D\subset \{g\in R_m\,|\, fg\in Rf\}= {\rm Nuc}_r(S_f)$.

 Conversely, if $D\subset {\rm Nuc}_r(S_f)$ then for all $a\in D$ there is $g(t)\in R$ such that
 $f(t)a=g(t)f(t)$.
 Comparing degrees, this means $g(t)\in D$, so that for all $a\in D$ there is $b\in D$ such that
 $f(t)a=bf(t)$.
\end{proof}

\begin{corollary} \label{cor:new(1)}
 Let $R=D[t;\delta]$. If $R$ is simple then there are no non-associative algebras $S_f$ such that
 $D\subset {\rm Nuc}_r(S_f).$
\end{corollary}

\begin{proof}
 If $R=D[t;\delta]$, $R$ is not simple if and only if there is a non-constant
semi-invariant $f\in R$  \cite{LLLM}. The assertion now follows from Proposition \ref{prop:new(1)}.
\end{proof}

We will assume throughout the paper that $f(t)\in D[t;\delta]$ has $\text{deg}(f(t)) = m \geq
2$ (if $f$ has degree $m=1$ then $S_f\cong D$) and that
$\delta\not=0$. Without loss of generality, we will only only look at monic $f(t)$.

\subsection{The characteristic $p>0$ case} For a division ring $D$ of characteristic $p$ and $R=D[t;\delta]$,
$$(t-b)^p=t^p-V_p(b), \quad V_p(b)=b^p+\delta^{p-1}(b)+*$$
 for all $b\in D$,
with $*$ a sum of commutators of $b$, $\delta(b),\dots, \delta^{p-2}(b)$.
If $D$ is commutative, or if $b$ commutes with all its derivatives, then the sum $*$ is 0 and the formula simplifies to
$$V_p(b)=b^p+\delta^{p-1}(b)$$
 \cite[p.~17ff]{J96}. An iteration  yields
$$(t-b)^{p^e}=t^{p^e}-V_{p^e}(b)$$
for all $b\in D$ \cite[1.3.22]{J96} with $V_{p^e}(b)=V_p^e(b)=V_p(\dots (V_p(b))\dots )$.
For any $p$-polynomial
$$f(t)=a_0t^{p^e}+a_1t^{p^{e-1}}+\dots+a_et+d\in D[t;\delta]$$
we thus have
$$f(t)-f(t-b)=a_0V_{p^e}(b)+a_1V_{p^{e-1}}(b)+\dots+a_eb$$
for all $b\in D$ and define
$$V_f(b)=a_0V_{p^e}(b)+a_1V_{p^{e-1}}(b)+\dots+a_eb.$$


\begin{lemma} \label{le:Am}
(i) \cite[Lemmata 4]{Am2}
  $t^p-t-a\in D[t;\delta]$ is either irreducible or a product of commutative linear factors.
\\ (ii)  \cite[Lemmata 6]{Am2} $f(t)=t^p-t-d\in D[t;\delta]$ is irreducible if and only if $V_f(z)\not=0$
for all $z\in D$ which is equivalent to
$$V_p(z)-z\not=d$$
for all $z\in D$.
\\ (iii)  \cite{CB} In characteristic 3, $f(t)=t^3-ct-d\in D[t;\delta]$ is irreducible if and only if
$$V_3(z)-cz\not=d  \text{ and } V_3(z)-zc+\delta(c)\not=d$$
for all $z\in D$.
\end{lemma}

\begin{proof}
(iii) $f$ is irreducible if and only if it neither right nor left divisible by some linear factor
$t-z$, $z\in D$. Now $f(t)\not=g(t)(t-z)$ for all $z\in D$ is equivalent to $V_f(z)\not=0$ for all $z\in D$ \cite{J96}
(i.e. to $V_3(z)-cz\not=d$),
and $f(t)\not=(t-z)h(t)$ for all $z\in D$ is equivalent to $V_3(z)-zc-d+\delta(c)\not=0$
for all $z\in D$ by a straightforward calculation.
\ignore{
(ii)  $t^p-t-a\in D[t;\delta]$ is reducible if and only if, by (i), it is right divisible by a linear factor
$t-z$, $z\in D$. This is equivalent to $V_p(z)-z\not=a$ \cite[1.3.26]{J96}.

}
\end{proof}

%
%

\section{Petit's algebras from algebraic derivations}\label{sec:2}

Let $C$ be a field  of characteristic $p$ and $D$  a central division algebra over $C$
 of degree $n$ (we allow $D=C$ here).
Let  $\delta$ be a derivation of $D$, such that $\delta|_C$ is algebraic with minimum polynomial
$$g(t)=t^{p^e}+a_1t^{p^{e-1}}+\dots+ a_et\in F[t]$$
 of degree $p^e$, where $F={\rm Const}(\delta)$. 
 Then $g(\delta)=id_{d_0}$ is an inner
derivation of $D$ and we choose $d_0\in F$  so that $\delta(d_0)=0$ \cite[Lemma 1.5.3]{J96}. The center of $R=D[t;\delta]$ is $F[z]$ with
$z=g(t)-d_0$ and the two-sided $f\in D[t;\delta]$ are of the form $f(t)=uh(t)$ with $u\in D$
and $h(t)\in C(R)$ \cite[Theorem 1.1.32]{J96}. For all $a\in C$, define
$$V(a)=V_g(a)=V_{p^e}(a)+a_1V_{p^{e-1}}(a)+\dots+a_ea.$$
Then $V(a)\in F$ \cite{J37} and $V:C\longrightarrow F$ is a homomorphism of the additive groups
$C$ and $F$.
Moreover,
$$V(a)=0 \text{ if and only if }a=\delta(c)/c$$
 for some $c\in C$ (\cite{J37}, cf. also \cite[p.~2]{Hoe}.
$V$ can be seen as an additive analogue to the norm of a cyclic separable field extension.

 In particular, $\delta$ is a quasi-algebraic derivation on $D$
 in the sense of \cite{LLLM} and so $R=D[t;\delta]$ is not simple.
Theorem \ref{Properties of S_f petit} together with Proposition \ref{prop:subfield} and  Corollary \ref{cor:new(1)}  yields:

\begin{theorem}\label{thm:mainI}
Let $f\in D[t;\delta]$ have degree $m$.
\\ (i) $S_f$ is a unital algebra over $F$ of
dimension $m n^2 p^e$ and if $f$ is irreducible then $S_f$ is a division algebra over $F$. If $f$ is not two-sided then
its left and middle nucleus are $D$. $D$ is not contained in the right nucleus.
\\ (ii) If $f\in F[t]$ then ${\rm Nuc}_r(S_f)$ contains the subring
$$F[t]/(f) \cong F\oplus F t\oplus\dots\oplus Ft^{m-1}$$
which is a subfield of degree $mp^e$ over $F$ whenever $f(t)\in F[t]$ is irreducible.
\\ (iii) If $f(t)\in C[t;\delta]$, then $S_f$ contains
$C[t;\delta]/C[t;\delta]f$  as a subalgebra of dimension $mp^{e}$ over $F$.
\end{theorem}

When $S_f$ is not associative, any automorphism of $S_f$ extends an automorphism of $D$ since the left nucleus of an algebra
 is left invariant under automorphisms.

 Let $H:D[t;\delta]\longrightarrow D[t;\delta]$
be any $F$-automorphism of $R=D[t;\delta]$. Then $H$ canonically induces an isomorphism of $F$-algebras
$$S_f\cong S_{H(f)}.$$
This leads us to:

 \begin{proposition}
Let $f(t)=a_0t^{p^e}+a_1t^{p^{e-1}}+\dots+a_et+d\in D[t;\delta]$ be a $p$-polynomial of degree $p^e$.
\\ (i) $S_f\cong S_h$ for all $h(t)=f(t)-V_f(a)$, $a\in C$.
\\ (ii) The map $G_{id,-a}$ defined via $G|_D=id_D$ and $G(t)=t-a$ is an automorphism of $S_f$ for all
$a\in C$ such that $V_{f}(a)=0.$
 \end{proposition}

 \begin{proof}
 We know that $G=G_{id,-a}$ is an automorphism of $R$ if and only if $a\in C(D)=C$.
 $G$ is $F$-linear. Since $f(t)-f(t-a)=V_f(a)$ for all $a\in C$,
$G(f(t))=f(t-a)=f(t)-V_f(a)$,
so that $G(f(t))=f(t)$ for all $a\in C$  with $V_{f}(a)=0$ implying (ii).
\end{proof}

We conclude:

\begin{proposition}\label{prop:1}
For $f(t)=g(t)-d\in D[t;\delta]$,
$${\rm ker}(V)=\{a\in C\,|\, V(a)=0 \}=\{\delta(c)/c\,|\, c\in C \}$$
 is isomorphic to the subgroup $\{G_{id,-a}\,|\, a\in C \text{ with }
V(a)=0\}$ of ${\rm Aut}_F( S_f)$.
\end{proposition}

\begin{proof}
There is a one-one correspondence between the sets ${\rm ker}(V)$ and $\{G_{id,-a}\,|\, a\in C,
V(a)=0\}$ of ${\rm Aut}_F( S_f)$ given by $a\mapsto G_{id,-a}$ which yields the assertion.
\end{proof}

For $f(t)=t^p-t-d$ we have in particular
$G(f(t))=f(t)$ for all $a\in C(D)$ with $\delta^{p-1}(a)+a^p-a=0$.

\begin{lemma}\label{le:iso_order_p}
For $f(t)=t^p-t-d\in D[t;\delta]$, $G_{id,-1}\in {\rm Aut}(S_f)$ has order $p$.
\end{lemma}

\begin{proof}
$G=G_{id,-1}$ is an automorphism of $R$ of order $p$ \cite{Am2}. For $f(t)=t^p-t-d$ we have
$G(f(t))=f(t)$ since $\delta^{p-1}(1)+1^p-1=0$.
Thus $G_{id,-1}$ induces an automorphism of $S_f$, it is easy to see it has order $p$.
\end{proof}

\ignore{
If $K$ is a field of characteristic $p$ then it has the subfield $\mathbb{F}_p$
and $x^p-x=0$ for all $x\in \mathbb{F}_p$.

\cite[Lemma 6]{Am2}  $t^p-t-a\in D[t;\delta]$ is irreducible if and only if for all $z\in D$, $$V_p(z)\not=a.$$

Problem with inconsistency between \cite{J96} and \cite{Am2} here in (ii):

 $d-cz\not=V_p(z)$  for $f(t)=t^p-ct-d$ in Jacobson \cite{J96},

$d\not=V_p(z)$ in \cite{Am2} for $f(t)=t^p-t-d$, instead of $d-z\not=V_p(z)$
}


%
%
%

\section{Nonassociative differential extensions of a field}\label{sec:3}

 Let $K$ be a field of characteristic $p$ together with a derivation $\delta:K\to K$
and $F={\rm Const}(\delta)$. Put $R=K[t;\delta]$.  We assume that $\delta$ is an algebraic derivation
of $K$ of degree $p^e$ with minimum polynomial
$$g(t)=t^{p^e}+c_1t^{p^{e-1}}+\dots+ c_et\in F[t]$$
 of degree $p^e$. Then $K$ is a purely inseparable extension of $F$ of exponent one, and
$K^p\subset F\subset K$.
 More precisely, $K=F(u_1, \dots,u_e)=F(u_1)\otimes_F\dots \otimes_F F(u_e)$, $u_i^p=a_i\in F$, and $[K:F]=p^e$.
 The center of $R$ is $F[z]$ with $z=g(t)-d_0$, $d_0\in F$.

Theorem \ref{thm:mainI} becomes:

\begin{theorem} 
\label{thm:9}
Let $f\in K[t;\delta]$ have degree $m$. Then $S_f$ is an algebra over $F$ of
dimension $mp^e$ and if $f(t)$ is irreducible then $S_f$ is a division algebra.
If $f$ is not two-sided then $S_f$ has left and middle nucleus $K$ and $K$ is not contained in the right nucleus.
\\ In particular, if $f(t)\in F[t]$ then ${\rm Nuc}_r(S_f)$ contains the subring
$$F[t]/(f) \cong F\oplus Ft\oplus\dots\oplus Ft^{m-1}$$
which is a subfield of degree $m$ over $F$ whenever $f(t)$ is irreducible in $F[t]$.
\end{theorem}


We will investigate the following special case:

\begin{definition} Let $f(t)=g(t)-d\in K[t;\delta]$.
Then the  $F$-algebra
$$(K,\delta,d)=S_f=K[t;\delta]/K[t;\delta] f(t)$$
is called a \emph{(nonassociative) differential extension} of $K$.
\end{definition}

$(K,\delta,d)$ has dimension $p^{2e}$, is free of dimension $p$ as a  $K$-vector space, and contains $K$
as a subfield.
$(K,\delta,d)$ is associative if and only if $d\in F$ and a division algebra if and only if $f(t)$ is irreducible.
For $d\in K\setminus F$ it has left and middle nucleus $K$.

\begin{proposition} \label{le:1}
(i) For $d\in K\setminus F$, the right nucleus of $(K,\delta,d)$ contains $K$, thus
$${\rm Nuc}((K,\delta,d))=K.$$
The powers of $t$ are not associative in $(K,\delta,d)$.
\\ (ii) For all $a,d\in K$, $(K,\delta,d)\cong (K,\delta,d-V(a)).$
\end{proposition}

\begin{proof}
(i) Since $g$ is semi-invariant and monic of minimal degree, we have $g(t)a=ag(t)$ for all $a\in K$ \cite[(2.1), p.~3]{LLLM},
 i.e. $f(t)a=af(t)$ for all $a\in K$ and
 so  $f$ is semi-invariant, too, and hence the right nucleus of $(K,\delta,d)$ contains $K$ by Proposition \ref{prop:new(1)}.
By \cite{LLLM}, $f$ is two-sided iff $f$ is semi-invariant and $ft\in Rf$. Here, $f$ is not two-sided, therefore
$ft\not \in Rf$, which implies that the powers of $t$ are not associative in $(K,\delta,d)$ by Theorem
\ref{Properties of S_f petit}.
\\ (ii) For $a,d\in K$ and $G=G_{id,-a}$ we have
$G(f(t))=f(t)- V(a)$ and $S_f\cong S_{G(f)}$.
\end{proof}

In fact, for $f(t)=g(t)-d\in F[t]$ (i.e. here $f(t)$ is two-sided),
$$(K,\delta,d)=K[t;\delta]/K[t;\delta] f(t)$$
is an associative central simple $F$-algebra called a \emph{differential extension of $K$} and treated in \cite[p.~23]{J96}.
 Then $K$ is a maximal subfield of $(K,\delta,d)$. 
Note that  $(K,\delta,d)$ contains the subring $F[t]/(f(t))$,
which is a field extension of $F$ of degree $p^e$ whenever $f(t)$ is irreducible in $F[t]$ (thus a maximal subfield
of $(K,\delta,d)$).

\begin{remark}\label{re:9}
 In the special case where $g(t)=t^p-t$ and $f(t)=t^p-t-d\in F[t]$, the automorphism group of
 $(K,\delta,d)$ has a cyclic subgroup of order $p$  generated by $G_{id,-1}$ which leaves $K$ invariant.
 If $f(t)=t^p-t-d$ is irreducible then the division algebra
 $(K,\delta,d)$ is called a \emph{cyclic extension of $K$} of degree $p$ by Amitsur,
 as it can be seen as a noncommutative generalization of
 a cyclic field extension of $K$: it  has dimension $p$ as a $K$-vector space and
 the automorphism group of $(K,\delta,d)$ has a cyclic subgroup of order $p$.
 All cyclic extensions of $K$ of degree $p$ are of this form  \cite{Am2}.
Note that they always contain the cyclic separable field extension $F[t]/(t^p-t-d)$ of degree $p$.
\end{remark}

\ignore{
  The index (for $A=Mat_n(D)$, index of $A$ is degree of $D$) of the differential extension $(K,\delta,d)$ is the smallest power $p^q$ such that
  $(d-z)^{p^q}=n(g)$ for some $g$ \cite[p.~31]{J96}.
\fbox{out or explain?}

}

 When $g(t)=t^p-t\in F[t]$ and $f(t)=t^p-t-d\in K[t;\delta]$, $d\in K\setminus F$ is irreducible,
the nonassociative division algebra $(K,\delta,d)$ is a canonical, nonassociative, generalization of
a cyclic extension:

 \begin{theorem} \label{thm:nonasscyclic}
 Let  
 $\delta$ be of degree $p$ with minimum polynomial $g(t)=t^p-t\in F[t].$
Let $f(t)=t^p-t-d\in K[t;\delta]$. Then
$$(K,\delta,d)=K[t;\delta]/K[t;\delta] f(t)$$
 is a nonassociative algebra over $F$ of dimension $p^{2}$, and is a division algebra if and only if
$$V_p(z)-z\not=d$$
for all $z\in K$, if and only if
$$z^p+\delta^{p-1}(z)-z\not=d$$
 for all $z\in K$.
 ${\rm Aut}_F(S_f)$ has a cyclic subgroup of order $p$ generated by $G=G_{id,-1}$, i.e. $G|_K=id_K$.
\end{theorem}

\begin{proof}
Here  $f(t)=t^p-t-d\in K[t;\delta]$ is irreducible if and only if for all $z\in K$,
$V_p(z)-z\not=d$ by Lemma \ref{le:Am} (ii). Since $K$ is commutative, the second equivalence is clear.
The remaining assertion follows from Lemma \ref{le:iso_order_p}.
\end{proof}

If $d\in F$ then  $S_f=(K,\delta,d)$ is the cyclic extension of
$F$ of degree $p$ in Remark \ref{re:9}. As a corollary of Theorem \ref{thm:nonasscyclic} we obtain a canonical construction method for
\emph{nonassociative cyclic extensions} of $K$,
if we define these algebras as division algebras containing $K$ which are $K$-vector spaces of dimension $p$ and
 whose automorphism group contains a cyclic subgroup of order $p$ which leaves $K$ invariant:

\begin{theorem}\label{thm:strongest}
Let $\delta$ be of degree $p$ with minimum polynomial $g(t)=t^p-t\in F[t].$ For all
 $f(t)=t^p-t-d\in K[t;\delta]$ with $d\in K\setminus F$,
 $(K,\delta,d)$ is a unital nonassociative division algebra over $F$ of dimension $p^{2}$.
 Its left and middle nucleus is $K$, its right nucleus contains $K$, and its automorphism group  contains a cyclic subgroup of order $p$
which leaves $K$ invariant.
\end{theorem}

\begin{proof}
Suppose there is $z\in K$ such that $z^p+\delta^{p-1}(z)-z=d$. Apply $\delta$ to both sides to obtain
$\delta(z^p)+\delta^{p}(z)-\delta(z)=\delta(d)$, which means
 $\delta(z^p)=\delta(d)$ since $\delta^p=\delta$ here.
Now $\delta^p(z)=pz^{p-1}\delta(z)=0$ implies that $\delta(d)=0$ and hence the first assertion since
$d\not\in F$ by Lemma \ref{le:Am} (ii).
The right nucleus contains $K$ by Proposition \ref{le:1} and the remaining assertion follows from Theorem \ref{thm:nonasscyclic}.
\end{proof}



\begin{example}\label{ex:1}
Let $\delta$ have minimum polynomial $g(t)=t^p-t\in F[t].$
Let $x$ be an indeterminate and $\delta$ be the extension of $\delta$ to $K(x)$ via
$\delta(x)=0$. Clearly ${\rm Const}(\delta)=F(x)$ and $g(t)=t^p-t\in F(x)[t]$
is the minimal polynomial of the extended derivation $\delta$. Then  for all
$h(x)\in K(x)\setminus F(x)$, $f(t)=t^p-t-h(x)$ is irreducible and hence
$$(K(x),\delta,h(x))$$
 is a unital nonassociative division algebra over $F(x)$
of dimension $p^{2}$, and a nonassociative cyclic extension of $K(x)$.
This generalized \cite[Proposition 1.9.10]{J96}.
\end{example}

When $F$ has characteristic 3, using Lemma \ref{le:Am} and Proposition \ref{prop:1} we can generalize Theorem \ref{thm:strongest} slightly:

\begin{theorem}\label{thm:strongestII}
Let $F$ have characteristic 3 and $\delta$ be of degree $3$ with minimum polynomial $g(t)=t^p-ct\in F[t].$ For all
 $f(t)=t^p-ct-d\in K[t;\delta]$ with $d\in K\setminus F$,
 $(K,\delta,d)$ is a nine-dimensional unital nonassociative division algebra over $F$.
 Its left and middle nucleus is $K$, its right nucleus contains $K$ and
 $\{\delta(c)/c\,|\, c\in K \}$ is isomorphic to the subgroup $\{G_{id,-a}\,|\, a\in C \text{ with }
V(a)=0\}$ of ${\rm Aut}_F((K,\delta,d))$.
\end{theorem}

\begin{proof}
Suppose there is $z\in K$ such that $z^3+\delta^{2}(z)-cz=d$. Apply $\delta$ to both sides to obtain
$\delta(z^3)+\delta^{p}(z)-c\delta(z)=\delta(d)$.
Now $\delta^3(z)=0$ and  $\delta^3=c\delta$  implies that $0=\delta(d)$, a contradiction.
Next assume that that there is $z\in K$ such that $z^3+\delta^{2}(z)-cz+\delta(c)=d$. Apply $\delta$ to both sides to obtain
$\delta(z^3)+\delta^{p}(z)-c\delta(z)++\delta^2(c)=\delta(d)$, i.e. again that $0=\delta(d)$, a contradiction.
Thus $f$ is irreducible by Lemma \ref{le:Am} (iii).
The right nucleus contains $K$ by Proposition \ref{le:1} and the assertion follows.
\end{proof}

\begin{example}\label{ex:3}
Let $F$ have characteristic 3 and $\delta$ be of degree $3$ with minimum polynomial $g(t)=t^3-ct\in F[t].$
Let $x$ be an indeterminate and $\delta$ be the extension of $\delta$ to $K(x)$ via
$\delta(x)=0$. Clearly ${\rm Const}(\delta)=F(x)$ and $g(t)=t^3-ct\in F(x)[t]$
is the minimal polynomial of the extended derivation $\delta$. Then  for all
$h(x)\in K(x)\setminus F(x)$, $f(t)=t^3-ct-h(x)$ is irreducible and so
$$(K(x),\delta,h(x))$$
 is a unital nine-dimensional nonassociative division algebra over $F(x)$.
This again generalizes \cite[Proposition 1.9.10]{J96}.
\end{example}

%
%

\section{Nonassociative differential extensions of a division algebra} \label{sec:4}

\subsection{} Let $C$ be a field of  characteristic $p$ and $D$  a central division algebra over $C$ of degree $n$.
Let  $\delta$ be a derivation of $D$, such that $\delta|_C$ is algebraic with minimum polynomial
$$g(t)=t^{p^e}+c_1t^{p^{e-1}}+\dots+ c_et\in F[t]$$
 of degree $p^e$ and $F={\rm Const}(\delta)$ as in Section \ref{sec:2}.

 \begin{definition}
For all $f(t)=g(t)-d\in D[t;\delta] $, the $F$-algebra
$$(D,\delta,d)=S_f=D[t;\delta]/D[t;\delta] f(t)$$
  is called a \emph{(nonassociative) generalized differential algebra}.
\end{definition}

$(D,\delta,d)$ is a unital nonassociative algebra over $F$ of dimension $p^{2e}n^2$ and free
of rank $p^e$  as a left $D$-module, and contains $D$
as a subalgebra. For $d\in D\setminus F$ it has left and middle nucleus $D$.

\begin{lemma} \label{le:2}
For $d\in D\setminus F$, the right nucleus of $(D,\delta,d)$ does not contain $D$, thus ${\rm Nuc}((D,\delta,d))$
is properly contained in $D$.
\\ If $d\in C\setminus F$, then $(C,\delta|_C,d)$ is a subalgebra of $(D,\delta,d)$ and the
 right nucleus of $(D,\delta,d)$ contains $C$, thus $C\subset {\rm Nuc}((D,\delta,d))$.
\end{lemma}

\begin{proof}
 Since $g$ is semi-invariant and monic of minimal degree, we have $g(t)a=ag(t)$ for all $a\in D$ \cite[(2.1), p.~3]{LLLM},
 i.e. $f(t)a=ag(t)-da$ for all $a\in D$ and so $f$ is not semi-invariant, since this would mean that
$da=ad$ for all $a\in D$ and we assumed $d\in D\setminus F$.
Hence the right nucleus of $(K,\delta,d)$ does not contain $D$ by Proposition \ref{prop:new(1)}.

If $d\in C\setminus F$, then $f\in C[t,\delta]=C[t;\delta|_C]$ is semi-invariant in  $C[t,\delta]$ and
$\delta|_C$ is an algebraic derivation on $C$ with minimum polynomial $g(t)$
 of degree $p^e$. Since $d\in C$, $f$ is semi-algebraic in $C[t,\delta]$, see the proof of Lemma \ref{le:1}.
 Thus for every $a\in C$ 
 we have $f(t)a\in C[t;\delta] f\subset Rf$ and hence $C\subset 
 {\rm Nuc}_r((D,\delta,d))$.
\end{proof}

\begin{proposition} \label{le:cong}
 For all $d\in D$ and $a\in C$,
 $$(D,\delta,d)\cong (D,\delta,d-V(a)).$$
\end{proposition}

\begin{proof} The proof is analogous to the one of Proposition \ref{le:1} (ii).
\end{proof}

$(D,\delta,d)$ is associative if and only if $d\in F$ and a division algebra if and only if $f(t)$ is irreducible.
For $f(t)=g(t)-d\in F[t]$, the associative $F$-algebra
$$(D,\delta,d)=S_f=D[t;\delta]/D[t;\delta] f(t)$$
   is a central simple algebra over $F$
and called a \emph{generalized differential extension of $D$} in \cite[p.~23]{J96}.
The defining relations characterizing the associative algebra $(D,\delta,d)$  are given by
$$ta=at+\delta(a) \text{ and } t^{p^e}+c_1t^{p^{e-1}}+\dots+ c_et=d$$
for all $a\in D$  \cite[p.~23]{J96}. Moreover, the central simple
algebra $(D,\delta,d)$ contains $D$ as the centralizer of $C$ \cite[Theorem 3.1]{Hoe} and
Proposition \ref{le:cong}  for $d\in F$ was proved in \cite[Theorem 3.2]{Hoe}.

In the special case where  $g(t)=t^p-t$ and hence $f(t)=t^p-t-d\in F[t]$,
the automorphism group of the central simple
algebra $(D,\delta,d)$ of degree $n^2p^2$ has a cyclic subgroup of order $p$ generated by $G_{id, -1}$ which leaves $D$ invariant.
If $f$ is irreducible then the division algebra $(D,\delta,d_0)$ is also called a \emph{cyclic extension of $D$}
of degree $p$ by Amitsur,
 as it is also free of rank $p$ as a right $D$-module and thus can be seen as canonical generalization of
 a cyclic field extension. All cyclic extensions of $D$ of degree $p$ are of this form \cite{Am2}.

Note that if $f(t)=t^p-t-d\in F[t]$ is irreducible, then $(D,\delta,d)$ also
 contains the cyclic field extension $F[t]/(t^p-t-d)$ of dimension $p$ over $F$ as a subfield.

\begin{theorem}\label{thm:last}
Let $\delta$ have minimum polynomial
$$g(t)=t^p-t\in F[t].$$
Then for all $f(t)=t^p-t-d\in D[t;\delta]$,
$$(D,\delta,d)=D[t;\delta]/D[t;\delta]f(t)$$
is a nonassociative algebra over $F$ of dimension $n^2p^2$ and a division algebra  if and only if
$$d\not=V_p(z)-z$$
for all $z\in D$, if and only if
$$d\not=(t-z)^p-t^p-z$$
 for all $z\in D$. $(D,\delta,d)$ is associative if and only if $d\in F$.

 ${\rm Aut}_{F}((D,\delta,d))$ has a cyclic subgroup of order $p$ generated by $G=G_{id,-1}$, i.e. $G|_D=id_D$.
\end{theorem}

\begin{proof}
We know that $S_f=(D,\delta,d)=D[t;\delta]/D[t;\delta]f(t)$ with $f(t)=t^p-t-d$
is a division algebra if and only if $d\not=V_p(z)-z$
 for all $z\in D$ by Lemma \ref{le:Am} (ii). The remaining assertion follows from Lemma \ref{le:iso_order_p}.
\end{proof}

This nicely generalizes \cite[Theorem 1.3.27]{J96} on cyclic extensions of $D$ whenever $f$ is irreducible.
We thus call unital nonassociative division algebras which contain $D$ as a subalgebra, are free of rank $p$
as a left $D$-module and have a cyclic subgroup  of automorphisms  of order $p$ which restrict to
$id_D$ on $D$, \emph{nonassociative cyclic extensions of $D$ of degree $p$}.

 In particular, if $f(t)\in C[t;\delta]$ in Theorem \ref{thm:last}, then $(D,\delta,d)$ contains the nonassociative
 cyclic extension
$(C,\delta,d)=C[t;\delta]/C[t;\delta]f$ of $C$ treated in Theorem \ref{thm:nonasscyclic} as a subalgebra of
dimension $p^{2}$ over $F$. This is a division subalgebra whenever $d\in C\setminus F$ by Theorem \ref{thm:strongest}.

Note also that for  $f(t)=t^p-t-d\in D[t;\delta]$ and all $a\in C$ we have
$$(D,\delta,d)\cong (D,\delta,d+\delta^{p-1}(a)+a^p-a)=(D,\delta,d+V_p(a)-a).$$

\begin{remark}
 Petit's construction of nonassociative algebras $S_f$ can be generalized to the setting where
$f\in S[t;\delta]$ is a monic polynomial and $S$ any unital associative ring \cite{P15.1}. Therefore
 some of the results
above also hold for nonassociative algebras obtained by employing $f(t)=t^p-t-d\in S[t;\delta]$ if $\delta$ satisfies
the polynomial identity $\delta^p=\delta$ as before. I.e., we can construct algebras which are free of rank $p$ as left $S$-modules whose automorphism group contains a cyclic subgroup of order $p$.
Amitsur's method of determining the (associative)
cyclic extensions of division rings $D$ was extended to simple rings $S$ already  in \cite{K}.
\end{remark}

When $F$ has characteristic 3, we can generalize Theorem \ref{thm:last} slightly, employing Lemma \ref{le:Am}
and and Proposition \ref{prop:1}:

\begin{theorem}\label{thm:lastII}
Let $\delta$ have minimum polynomial
$$g(t)=t^3-ct\in F[t].$$
Then for all $f(t)=t^3-ct-d\in D[t;\delta]$,
$$(D,\delta,d)=D[t;\delta]/D[t;\delta]f(t)$$
is a nonassociative unital algebra over $F$ of dimension $9n^2$ and a division algebra  if and only if
$$V_3(z)-cz\not=d  \text{ and } V_3(z)-zc-d+\delta(c)\not=0$$
 for all $z\in D$. $(D,\delta,d)$ is associative if and only if $d\in F$.
${\rm Aut}_{F}((D,\delta,d))$ has a subgroup isomorphic to $\{\delta(c)/c\,|\, c\in K \}$.
\end{theorem}

\subsection{}

 Let $D$ be a central division algebra over $C$ of degree $[D:F]=n$ and let $C$ have characteristic $p$. As a consequence
 of  \cite[(3)]{P66} we obtain
 the following partial generalization of \cite[Theorem 3]{Am2} which states when a nonassociative
 cyclic extension $S$ of $D$ of degree $p$ has the form discussed in Theorem \ref{thm:last}:

\begin{theorem} \label{thm:class}
 Let $S$ be a division ring with multiplication  $\circ$, which is not associative, such that
\\ (1)  $S$ has $D$ as subring, is a free left $D$-module of rank $p$, and there is $t\in S$ such that
 $t^0,t,t^2,\dots,t^{p-1}$ is a basis of $S$ over $D$, when defining $t^{i+1}=t\circ t^{i}$, $t^0=1$, for $0\leq i<p$;
\\ (2) for all $a\in D$, $a\not=0$, there is $a'\in D^\times$ such that $t\circ a=a\circ t+a'$;
\\ (3) for all $a,b,c\in D$, $i+j<p$, $k<p$, we have
$$[a\circ t^i,b\circ t^j,c\circ t^k]=0,$$
\\ (4) $t^p=t+d$ for some $d\in D^\times$ with  $t^{p}=t\circ t^{p-1}$ as above.
 Then
  $$\delta(a)=a'=t\circ a-a\circ t \quad (a\in D)$$
 is a derivation on $D$ and
  $$S\cong S_f$$
  with $f(t)=t^p-t-d \in D[t;\delta]$ irreducible.
 For any $H\in {\rm Aut}_{F}(S_f)$, $H|_D\in {\rm Aut}_{F}(D)$.
\\
If, in particular, $\delta|_C$ is algebraic with minimal polynomial $g(t)=t^p-t$ and $F={\rm Const}(\delta)$
 then $S$ is a nonassociative cyclic extension of $D$ of dimension $p^2[D:F]$ over $F$.
\end{theorem}

\begin{remark}
 Conditions (1), (2), (3) are equivalent to conditions (1), (2), (5), (6), (7)  with
\\ (5) $D\subset {\rm Nuc}_l(S)\cap {\rm Nuc}_m(S)$;
\\ (6) $t^i\circ b=t\circ(t^{i-1}\circ b)$ for all $b\in D$, $0\leq i<p$,
\\ (7) for $0\leq i,j,k<p$ and $i+j<p$, $k<p$, we have $[t^i,t^j,t^k]=0$ \cite[(3)]{P66}.
\end{remark}

An analogous result holds when $D=K$ is a  field of characteristic $p$ and we consider the setup as in Section \ref{sec:3}.

%
%

\section{Some tensor product constructions}\label{sec:tensor}

Let $E/F$ be a finite dimensional purely inseparable extension of exponent one and characteristic $p$
and $\delta$ a derivation on $E$ such that $F={\rm Const}(\delta)$. Then
$\delta$ is an algebraic derivation of degree $p^e$ with minimum polynomial
$$g(t)=t^{p^e}+c_1t^{p^{e-1}}+\dots+ c_et\in F[t]$$
of degree $p^e$, and $[E:F]=p^e$.

Let $D$ be an (associative) central division algebra over $F$
 such that $D_E=D\otimes_F E$ is a division algebra
 and let $\delta$ be the extension of $\delta$ to $D_E$ such that $\delta|_D=0$.
 Suppose that
$$S_f=E[t;\delta]/E[t;\delta]f(t)$$
with $f(t)\in E[t;\delta]$ of degree $m$, is a division algebra of dimension $mp^e$ over $F={\rm Const}(\delta)$
(i.e., that $f(t)\in E[t;\delta]$ is irreducible). Then the tensor product
$$S_f\otimes_F D=E[t;\delta]/E[t;\delta]f(t)\otimes_F D \cong D_E[t;\delta]/D_E[t;\delta]f(t)$$
is a nonassociative  algebra over $F$ 
 of dimension  $m p^{e} [D:F]$ and a division algebra if and only if $f(t)$ is irreducible in $D_E[t;\delta]$.
 We consider the following special case:

\begin{theorem} \label{thm:Jacobsontensor}
If $g(t)=t^p-t\in F[t]$ is  the minimal polynomial of $\delta$ and $f(t)=t^p-t-d\in E[t;\delta]$, then
$$(E,\delta,d)\otimes_F D\cong D_E[t;\delta]/D_E[t;\delta]f(t)$$
and
$$(E/F,\delta,c)\otimes_F D$$
is a  nonassociative division algebra over $F$  of dimension  $ p^{2} [D:F]$  if and only if
$$d\not=V_p(z)-z$$
  for all $z\in D_E$, if and only if
$$d\not=(t-z)^p-t^p-z$$
  for all $z\in D_E$.

   ${\rm Aut}_{F}((E/F,\delta,c)\otimes_F D)$ has a cyclic subgroup of order $p$ generated by $G=G_{id,-1}$, i.e. $G|_D=id_D$.
\end{theorem}

\begin{proof}
$f(t)=t^p-t-d$ is irreducible if and only if $d\not=V_p(z)-z$ for all $z\in D_E$ by Lemma \ref{le:Am} (ii).
This is equivalent to $d\not=t^p-(t-z)^p-z$ for all $z\in D_E$ by \cite[(1.3.19)]{J96}.
 The remaining assertion follows from Lemma \ref{le:iso_order_p}.
\end{proof}

This generalizes \cite[Theorem 1.9.13]{J96} which  appears as the case  $d\in F$.

\begin{example}\label{ex:last}
Let $\delta$ be of degree $p$ with minimum polynomial $g(t)=t^p-t\in F[t].$
Let $x$ be an indeterminate and $\delta$ be the extension of $\delta$ to $K(x)$ via
$\delta(x)=0$, where ${\rm Const}(\delta)=F(x)$.
For all $f(t)=t^p-t-h(x)$ with $h(x)\in K(x)\setminus F(x)$,
$(K(x),\delta,h(x))$ is a unital nonassociative division algebra over $F(x)$
of dimension $p^{2}$, and a nonassociative cyclic extension of $K(x)$, see Example \ref{ex:1}.

Let $D$ be a central division algebra over $F$ of degree $n$ such that $D\otimes_F K$ is a division algebra.
Then
$$(K(x),\delta,h(x))\otimes_{F(x)}D_{F(x)}\cong D_{K(x)}[t;\delta]/D_{K(x)}[t;\delta]f(t)$$
is a nonassociative unital algebra over $F(x)$ of dimension $p^2 n^2$ and a division algebra if and only if
$$h(x)\not=V_p(z)-z$$
for all $z\in D_{K(x)}$, if and only if
$$h(x)\not=(t-z)^p-t^p-z$$
for all $z\in D_{K(x)}$.

\ignore{
note that for all $z\in K(x)$, the center of $D_{K(x)}$,
 we have $h(x)\not=(t-z)^p-t^p-z$, since $h(x)\not\in F(x)$ by our choice.
}

Its automorphism group has a cyclic subgroup of order $p$ generated by $G=G_{id,-1}$, so that
the algebra is a nonassociative cyclic
extension of $D_{K(x)}$ if it is division.

This can be seen as a generalization of \cite[Theorem 1.9.11]{J96}, where $h(x)=x$
in which case the algebra is division.
\end{example}


%
%
%


\end{document}